\documentclass[12pt]{amsart}
\usepackage{pdfsync}
\usepackage{graphicx,color}
\usepackage{amssymb}
\usepackage{amsmath}
\usepackage{amsthm}
\usepackage{enumerate}
\usepackage{subcaption}
\usepackage{esvect}
\usepackage[left=3cm,top=3.8cm,right=3cm]{geometry} 
\usepackage{ucs}
\usepackage{cite}
\usepackage{amsxtra}
\usepackage{epstopdf}
\usepackage{bbold}
\usepackage[utf8x]{inputenc}
\usepackage{verbatim}                                  
\usepackage[toc,page]{appendix}                        
\usepackage{pdfsync} 

\addtocontents{toc}{\protect\setcounter{tocdepth}{1}}

\newcommand{\dist}{\text{\rm dist}}

\DeclareMathOperator{\diam}{diam}
\DeclareMathOperator{\rad}{rad}

\newtheorem{theorem}{Theorem}
\newtheorem{definition}[theorem]{Definition}
\newtheorem{proposition}[theorem]{Proposition}
\newtheorem{corollary}[theorem]{Corollary}
\newtheorem{observation}[theorem]{Observation}
\newtheorem{lemma}[theorem]{Lemma}

\newtheorem{remark}[theorem]{Remark}

\newcommand{\R}{\mathbb{R}}
\def\Z{{\mathbb Z}}
\def\N{{\mathbb N}}

\makeatletter
\def\moverlay{\mathpalette\mov@rlay}
\def\mov@rlay#1#2{\leavevmode\vtop{%
   \baselineskip\z@skip \lineskiplimit-\maxdimen
   \ialign{\hfil$\m@th#1##$\hfil\cr#2\crcr}}}
\newcommand{\charfusion}[3][\mathord]{
    #1{\ifx#1\mathop\vphantom{#2}\fi
        \mathpalette\mov@rlay{#2\cr#3}
      }
    \ifx#1\mathop\expandafter\displaylimits\fi}
\makeatother

\definecolor{Azul}{rgb}{0.0, 0.0, 1.0}
\definecolor{Rojo}{rgb}{1.0, 0.03, 0.0}
\definecolor{Purpura}{rgb}{0.5,0,0.35}
\definecolor{BurntOrange}{cmyk}{0,0.51,1,0}
\definecolor{PineGreen}{cmyk}{0.92,0,0.59,0.250}
\definecolor{Violeta}{rgb}{0.39,0.17,0.63}
\definecolor{Violeta2}{rgb}{0.6, 0.4, 0.8}
\definecolor{Fucsia}{rgb}{1.0, 0.01, 0.24}
\definecolor{Rosa}{rgb}{0.63,0.17,0.39}
\definecolor{VerdeManzana}{rgb}{0.39,0.63,0.17}
\definecolor{Celeste}{rgb}{0.7,0.7,1}

\begin{document}
\thispagestyle{empty}

\title[intersections of thick compact sets of $\R^d$]{intersections of thick compact sets in $\R^d$}

\author{Kenneth Falconer}
\address{ School of Mathematics and Statistics\\
University of St. Andrews}
\email{kjf@st-andrews.ac.uk}

\author{Alexia Yavicoli}
\address{ School of Mathematics and Statistics, University of St. Andrews, UK.
Current address: Department of Mathematics, the University of British Columbia. 1984 Mathematics Road, Vancouver BC V6T 1Z2, Canada}
\email{yavicoli@math.ubc.ca, alexia.yavicoli@gmail.com}


\begin{abstract}
We introduce a definition of thickness in $\R^d$ and obtain a lower bound for the Hausdorff dimension of the intersection of finitely or countably many thick compact sets using a variant of Schmidt's game. As an application we prove that given any compact set in $\R^d$ with thickness $\tau$, there is a number $N(\tau)$ such that the set contains a translate of all sufficiently small similar copies of every set in $\R^d$  with at most $N(\tau)$ elements; indeed the set of such translations has positive Hausdorff dimension. We also prove a gap lemma related result and bounds relating Hausdorff dimension and thickness.
\end{abstract}


\maketitle

\tableofcontents


\section{Introduction.}

The classical co-dimension formula states that if $C_1,C_2$ are submanifolds of $\R^d$ that intersect transversally then
\begin{equation}\label{codim}
\dim (C_1 \cap C_2)=\dim(C_1)+\dim(C_2)-d
\end{equation}
provided the right hand side is non-negative, where $\dim$ denotes the dimension of the manifolds. There are various versions of \eqref{codim} that are applicable in other settings, in particular for more general sets  using Hausdorff dimension $\dim_H$.
For example, for compact sets $C_1,C_2\subset\R^d$
\begin{equation}\label{codim2}
\dim_H (C_1 \cap (C_2+x) )\leq \max \{0, \dim_H(C_1\times C_2)-d\}
\end{equation}
for Lebesgue almost-all $x\in \R^d$; the right-hand side can be replaced by $\max \{0, \dim_H(C_1)+\dim_H(C_2)-d\}$ if, for example, either  $C_1$ or $C_2$ has equal Hausdorff and upper box-counting dimension, see \cite{MatP}.
On the other hand, for all $\epsilon>0$,
\begin{equation}\label{codim3}
\dim_H (C_1 \cap \sigma(C_2))\geq \max \{0, \dim_H(C_1)+\dim_H(C_2)-d-\epsilon\}
\end{equation}
for a set of similarities $\sigma$ of positive measure with respect to the natural measure on the group of similarities $\sigma$ on $\R^d$. The similarity group may be replaced by the group of isometries if  $\dim_H C_1> (d+1)/2$ (it is not known if this condition is necessary if $d\geq 2$), see \cite{Kah, MatP}. The disadvantage of these results is that they are measure theoretic, and tell us nothing about which particular similarities or isometries these inequalities are valid for.

At the other extreme, there are classes $\mathcal{C}$ of `limsup sets' of Hausdorff dimension $0<s<d$ which are dense in $\R^d$  with the property that the intersection of any countable collection of similar copies of sets in $\mathcal{C}$ still has  Hausdorff dimension $s$, see for example \cite{Fal94}.

It is natural to ask for specific conditions on compact sets that are `close enough' to each other that guarantee non-empty intersection, or even give a lower bound for the dimension of their intersection. For subsets of the real line Newhouse \cite{Newhouse} introduced a notion of thickness, see Definition \ref{thickness},
which depends on the relative sizes of the complementary open intervals of the set and showed that two Cantor-like sets, with neither contained in a gap of the other, must intersect if the product of their thickness is greater than 1, see Theorem \ref{NGapLem}.

In this paper we propose a definition of thickness for compact subsets of $\R^d$ for all $d \geq 1$. We obtain a higher dimensional gap lemma result for some cases, and show that given several compact sets in $\R^d (d\geq 1)$ that are not too far apart in a sense that will be made precise, if their thicknesses are large enough then they have non-empty intersection, and we obtain a lower bound for the Hausdorff dimension of this intersection.

We first review the definition of thickness for subsets of the real line. Recall that every compact set $C$ on the real line can be constructed by starting with a closed interval $I\equiv I_1$ (the convex hull of $C$) and successively removing disjoint open complementary intervals (they are the path-connected components of the complement of $C$). Clearly there are finitely or countably many disjoint open complementary intervals $(G_n)_n$, which we may assume are ordered so that their lengths $|G_n|$ are non-increasing; if several intervals have the same length, we order them arbitrarily. The two unbounded path-connected components of $\mathbb{R}\setminus C$ are not included. For each $n \in \mathbb{N}$  the interval $G_{n}$ is a subset of some closed path-connected component $I_{n}$  of $I\setminus (G_1\cup\cdots\cup G_{n-1})$. We say  that such a $G_{n}$  is {\em removed}  from $I_{n}$.

\begin{definition}[Thickness in $\R$]\label{thickness}
Let $C \subset \R$ be compact with convex hull $I$, and let $(G_n)_n$ be the ordered sequence of open intervals comprising $I\setminus C$. Each $G_n$ is removed from a closed interval $I_n$, leaving behind two closed intervals $L_n$ and $R_n$; the left and right intervals of $I_n \setminus G_n$.
The {\em thickness} of $C \subset \R$ is defined as
\[
\tau (C):= \inf_{n \in \N} \frac{\min \{ |L_n| , |R_n| \}}{|G_n|}.
\]
The sequence of complementary intervals $(G_n)_n$ may be finite, in which case the infimum is taken over the finite set of indices.

The thickness of a single point is taken to be  $0$, and that of a non-degenerate interval to be $+\infty$.
\end{definition}
If there are several complementary intervals of equal length, then the ordering of them does not affect the value of $\tau(C)$. See \cite{HKY93,Astels,PT93,AY} for more information on Newhouse thickness and alternative definitions.

\begin{theorem}[Newhouse's Gap Lemma]\label{NGapLem}
Given two compact sets $C_1, C_2 \subset \R$, such that neither set lies in a gap of the other, if \textbf{$\tau(C_1) \tau(C_2) > 1$} then
\[
C_1 \cap C_2 \neq \emptyset.
\]
\end{theorem}
Theorem \ref{NGapLem} was proved only for subsets of  $\R$ and it does not guarantee positive Hausdorff dimension of the intersection, nor does it generalise in any simple way to intersections of $3$ or more sets.

Here we give a definition of thickness for compact subsets of $\R^d$ that enables us to generalize Theorem \ref{NGapLem} to higher dimensions, and also obtain lower bounds for the Hausdorff dimension of the intersection of several sets. For a different definition of thickness for certain dynamically defined subsets of the complex plane  see \cite{Biebler}.

Our setting throughout the paper is as follows. Given a compact subset $C$ of $\R^d$, we define $(G_n)_{n=1}^\infty$ to be the (at most) countably many open bounded path-connected components of $C^C$  and $E$ to be the unbounded open path-connected component of $C^C$ (except when $d=1$ when $E$ consists of two unbounded intervals). We  call $E$ together with $G_n\ (n \in \N)$  the {\em gaps} of $C$. We may assume that the sequence of gaps $(G_n)_{n=1}^\infty$ is ordered by non-increasing diameter. Note that we make no assumption about the connectedness or simply connectedness of $C$.

We write $\dist$ for the usual distance in $\R^d$.

\begin{definition}[Thickness in $\R^d$]\label{thicknessRd}
We define the {\em thickness} of $C$ to be
\[\tau(C):=\inf_{n \in \N} \frac{\dist (G_n, \bigcup_{1\leq i \leq n-1}G_i \cup E)}{\diam (G_n)},\]
provided that $E$ is not the only path-connected component of $C$.

When the only complementary path-connected component is $E$, we define
\begin{equation}
\tau (C):= \left\{ \begin{array}{lcc}
             +\infty &   \text{if}  & C^{\circ} \neq \emptyset \\
             \\ 0 &  \text{if}  & C^{\circ} = \emptyset
             \end{array} \right .
 \end{equation}

We say $C$ is {\em thick} if  $\tau (C) >0$.
\end{definition}

If the sequence of complementary intervals $(G_n)_n$ is finite then the infimum is taken over the finite set of indices.
Moreover, thickness is well-defined in the sense that if two gaps have the same diameter, interchanging their positions in the ordering does not change the definition of thickness.

Note that $\tau \in [0, +\infty]$. Also, $\tau$ is invariant under homothetic maps, and agrees with the usual definition of thickness in the real line (recall Definition \ref{thickness}).

\begin{observation}
If $C\subset \R^d$ is a thick compact set, then either there are finitely many gaps  $(G_n)_n$ or $\lim_{n \to \infty} \diam G_n=0$.
 To see this we can assume that $E$ is not the only complementary path-connected component. If   $\diam G_n\geq c>0$ for infinitely many $n$, taking points $x_n \in G_n$ with $|x_n-x_i| \geq c\tau(C)$ for $1\leq i <n$ contradicts the sequential compactness of $E^C$.
\end{observation}

In Section \ref{SectionGapLemma}, we obtain a higher dimensional gap lemma related result. The gap lemma does not generalize in any simple way to intersections of three or more sets, so we need to use other methods to study such intersections. To achieve this we obtain lower bounds for the Hausdorff dimension of the intersection of several thick compact sets in terms of their thicknessess, which is easy to estimate in many cases.

Our main theorem, Theorem \ref{THEOREMintersection} which will follow from Theorem \ref{TEOlowerbound}  which relates thickness to `winning sets'.

The following constants appear in many of our results:
\begin{definition}\label{DefConstants}
In $\R^d\, (d\geq 1)$, let
\begin{equation}\label{knumbers}
K_1:=\frac{2d (24\sqrt{d})^d  \log (16\sqrt{d})}{1-\frac{1}{2^d}} \quad \text{ and } \quad  K_2:=\bigg( \frac{(24 \sqrt{d})^d (1+4^d 2)}{1-\frac{1}{2^d}} \bigg)^2.
\end{equation}
\end{definition}

We now state our main theorem which will follow from applying Theorem \ref{TEOlowerbound}  on  `winning sets' to thickness. We write $E_i$ for the unbounded open path-connected component of $C_i^C$ (the union of two unbounded intervals when $d=1$).

\begin{theorem}[Intersection of compact sets in $\R^d$]\label{THEOREMintersection}
Let $(C_i)_i$ be a family of countably many compact sets in $\R^d$, where $C_i$ has thickness $\tau_i>0$,  such that:
\begin{enumerate}[(i)]
\item $\sup_i \diam (C_i) <+\infty$,
\item there is a ball $B$ such that $B \cap E_i =\emptyset$ for every $i$, where $E_i$ is the unbounded component of $C_i^C$,
\item there exists $c \in (0,d)$ such that
$$\sum_i \tau_i^{-c} \leq \frac{1}{K_2}\beta^c (1-\beta^{d-c})$$
 where
   \[\beta:=\min \Big\{\frac{1}{4}, \frac{\diam (B)}{\sup_i \diam(C_i)}\Big\}.\]
\end{enumerate}
Then \[\dim_H \Big(B \cap \bigcap_i C_i\Big)\ \geq\ d- K_1 \frac{\left( \sum_i \tau_i^{-c}\right)^{d/c}}{\beta^d |\log (\beta)|}>0.\]
\end{theorem}

Note that condition (iii) comes from  Theorem \ref{TEOlowerbound} and is needed  both to obtain the lower bound for the dimension of the intersection and to ensure that this bound is positive.

The significance of Theorem \ref{THEOREMintersection} is that  a condition on thicknesses can give a lower bound for the dimension of intersection  of a finite or countable collection of sets in $\R^d$  so ensure that the intersection is non-empty. In practice, the thicknesses needed are rather large as a consequence of the large constants  $K_1$ and $K_2$.

A very active research area involves finding conditions on a set that guarantees the set contains  homothetic copies of  a given finite set of points, called a {\em pattern} in this context. It will follow from Theorem \ref{THEOREMintersection} that a set contains homothetic copies of any given pattern in $\R^d$ provided it is sufficiently thick.
Patterns and intersections are related: the set $C$ contains a homothetic copy of $A:=\{a_1, \ldots, a_n\}$ if and only if there exists $\lambda \neq 0$ such that $\bigcap_{1 \leq i \leq n}(C-\lambda a_i)\neq \emptyset$.

A consequence of the Lebesgue density theorem is that any set $E\subset\R^d$ of positive Lebesgue measure contains a homothetic copy of every finite set at all sufficiently small scales, so it is natural to seek conditions on sets of zero Lebesgue measure form which this remains true.
Perhaps the most natural notion of size to consider is Hausdorff dimension but there are constructions (see for example \cite{Kel2, Kel1, Mat, AY17, DMT, UMAY}) which indicate that Hausdorff dimension cannot, in itself, detect the presence or absence of patterns in sets of Lebesgue measure zero, even in the most basic case of points in arithmetic progressions.

{\L}aba and Pramanik \cite{LP09} showed that if, in addition to having large Hausdorff dimension, a subset of $\R$ supports a probability measure with appropriate Fourier decay, then it contains arithmetic progressions of length $3$. The hypotheses were relaxed and the family of patterns covered greatly enlarged in subsequent papers \cite{HLP16, CLP16, FGP19}. This work uses harmonic analysis, and such methods do not work easily for longer arithmetic progressions. Moreover, the hypotheses may be difficult to check, and are not even known to hold for natural classes of fractals such as central self-similar Cantor sets.

Yavicoli \cite{AY} showed that Newhouse thickness, Definition \ref{thickness},
 allows the detection of homothetic and more general copies of patterns inside fractal sets in the real line. Newhouse thickness is  easy to compute or estimate for many classical fractal sets such as self-similar sets or sets defined in terms of continued fraction coefficients.
Our notion of thickness in higher dimensions, Definition \ref {thicknessRd}, enables such results to be extended to $\R^d$.

\begin{theorem}\label{TeoPatterns}
Let $C\subset\mathbb{R}^d$ be a compact set with thickness $\tau:=\tau(C)$, such that $E^C$ contains a ball $B$. Then $C$ contains a homothetic copy of every set $A$ with at most
\begin{equation}\label{ntau}
N(\tau):= \left\lfloor \frac{\beta^d  |\log(\beta)|}{e K_2} \frac{\tau^d}{\log(\tau)} \right\rfloor
\end{equation}
elements, where \[\beta:=\min \Big\{\frac{1}{4}, \frac{15 \diam (B)}{16\diam (C)}\Big\}.\]
and $K_2$ is as in \eqref{knumbers}.

Moreover, for all $\lambda \in \big(0,\frac{\diam (B)}{16\diam (A)}\big)$, there exists a set $X$ of positive Hausdorff dimension (depending on $A$, $B$, $C$ and $\lambda$) such that
\[x+\lambda A \subseteq C  \mbox{ for all } x\in X.\]
\end{theorem}

We also discuss the relationship between Hausdorff dimension and thickness of a set. It is shown in \cite[p.77]{PT93} that for $C\subset \R$,
\begin{equation}\label{TauDimBook}
\dim_H (C)\geq \frac{\log 2}{\log(2+1/\tau (C))},
\end{equation}
and in Section \ref{DimTau} we obtain analogous lower bounds for $C\subset\R^d$.

\section{A Gap Lemma related result in $\R^d$}\label{SectionGapLemma}

In this section we extend Theorem \ref{NGapLem}, Newhouse's gap lemma on $\R$,  to certain cases in $\R^d$ for $d\geq 2$.
We study a particular case when the gaps are either linked or  do not intersect; in this setting we  can use an analogous  argument to Newhouse's proof.

We denote the boundary of $U \subset \R^d$ by $\partial U$.

\begin{definition}
We say that $U \subseteq \R^d$ and $V\subseteq \R^d$ are {\em linked gaps} if $U \cap V \neq \emptyset$, $(\partial U) \setminus V \neq \emptyset$ and $(\partial V)\setminus U \neq \emptyset$.

We say that $C_1$ and $C_2$ are {\em  linked compact sets} in $\R^d$ if for every pair of gaps $G^{1}$ and $G^{2}$ of $C_1$ and $C_2$ respectively we have that either their intersection is empty or they are linked gaps.
\end{definition}

\begin{proposition}\label{UglyGapLemma}
Let $C_1$ and $C_2$ be linked compact sets in $\R^d$, with $\tau(C_1) \tau(C_2)>1$, then $C_1 \cap C_2 \neq \emptyset$.
\end{proposition}

\begin{proof}
By definition of $\tau$,
\[\tau_1:=\tau(C_1):=\inf_{m} \frac{\dist\big(G_m^{1}, \bigcup_{1\leq i \leq m-1}G_i^{1} \cup E_1\big)}{\diam (G_m^{1})}\]
and
\[\tau_2:=\tau(C_2):=\inf_{n} \frac{\dist\big(G_n^{2}, \bigcup_{1\leq i \leq n-1}G_i^{2} \cup E_2\big)}{\diam (G_n^{2})}\]
where $C_1$ and $C_2$  have gaps $G_n^{1}$ and $G_n^{2}$ and external path-connected components $E_1$ and $E_1$  respectively.

We  assume that $C_1 \cap C_2 = \emptyset$ and will obtain a contradiction. Then, \[C_1 \subseteq C_2^C=\bigcup_i G_i^{2} \cup E_2  \ \text{ and } \  C_2 \subseteq C_1^C=\bigcup_i G_i^{1} \cup E_1.\]
We will construct inductively a sequence $(U_i,V_i)_{i \in \N}$ of pairs of linked gaps that occur in the construction of $C_1$ and $C_2$ respectively, such that either $\diam U_i \to 0$ or $\diam V_i\to 0$ (or both).

{\em To start the induction:} Since $E_1 \cap E_2 \neq \emptyset$ and $C_1$ and $C_2$ are linked, $E_1$ and $E_2$ are linked gaps, so we take $(U_1,V_1):=(E_1, E_2)$.

{\em Inductive step:} Given that we have defined  a pair of linked gaps $(U_k, V_k)$ of $C_1$ and $C_2$ defined, we now define $(U_{k+1}, V_{k+1})$.

Since $(U_k, V_k)$ is a pair of linked gaps, there is $a_k \in \partial U_k \setminus V_k$. Since $a_k \in \partial U_k$, we have $a_k\in C_1$, hence by assumption $a_k \notin C_2$, so there is  a gap $G_{n_k}^{2}$ of $C_2$ such that $a_k \in G_{n_k}^{2}$.
Note that $(U_k, G_{n_k}^{2})$ are linked because they intersect and $C_1$ and $C_2$ are linked.

In the same way, since $(U_k, V_k)$ is a pair of linked gaps there is $b_k \in \partial V_k \setminus U_k$. Since  $b_k \in \partial V_k$, then $b_k\in C_2$, hence $b_k \notin C_1$, so there is $G_{m_k}^{1}$ a gap of $C_1$ such that $b_k \in G_{m_k}^{1}$
Again $(G_{m_k}^{1}, V_k)$ are linked.

We will show that we can choose  $(U_{k+1}, V_{k+1})$ to be either $(U_k, G_{n_k}^{2})$  or $(G_{m_k}^{1}, V_k)$ in such a way the diameters of either $U_k$ or $V_k$ tends to $0$.

We observe that for a fixed pair $n,m \in \N$ the following two inequalities cannot hold simultaneously:
\begin{itemize}
\item $\dist (G_{m}^{1}, \bigcup_{1\leq i \leq  m-1} G_i^{1} \cup E_1)\leq \diam (G_{n}^{2})$
\item $\dist (G_{n}^{2}, \bigcup_{1\leq i \leq n-1}G_i^{2} \cup E_2) \leq \diam (G_{m}^{1}).$
\end{itemize}
For if both hold, then by definition of thickness,
\[\diam (G_{n}^{2})\geq \tau_1 \diam (G_{m}^{1}) \ \text{ and }\  \diam(G_{m}^{1})\geq \tau_2 \diam (G_{n}^{2}).\]
Using the hypothesis that   $\tau_1 \tau_2>1$,
\[\diam(G_{m}^{1})\geq \tau_2 \diam (G_{n}^{2}) \geq \tau_1 \tau_2 \diam (G_{m}^{1})>\diam (G_{m}^{1}),\]
a contradiction.

The  gaps $U_k$ and $V_k$ can be identified as $U_k:=G_{m}^{1}$ and $V_k:=G_{n}^{2}$ for some $n,m \in \N$.
In the case $\dist (G_{m}^{1}, \bigcup_{1\leq i \leq  m -1} G_i^{1} \cup E_1)> \diam (G_{n}^{2})$,we also know that $(U_k, V_k)$ are linked, so $\overline{V_k}$ does not intersect $G_i^{1}$ for every $1\leq i \leq  m-1$. Also $b_k \in \partial V_k \setminus U_k$. Then $b_k \in G_{m_k}^{1}$ with $m_k>m$, and we take $(U_{k+1}, V_{k+1}):=(G_{m_k}^{1}, V_k)$.

In the case $\dist (G_{m}^{1}, \bigcup_{1\leq i \leq  m -1} G_i^{1} \cup E_1)\leq \diam (G_{n}^{2})$, by the previous observation we have $\dist (G_{n}^{2}, \bigcup_{1\leq i \leq n-1}G_i^{2} \cup E_2) > \diam (G_{m}^{1})$. Analogously to the previous case $a_k \in  G_{n_k}^{2}$ with $n_k>n$, and we take $(U_{k+1}, V_{k+1}):=(U_k, G_{n_k}^{2})$.

Since one or other of these cases occurs infinitely many times, we get a sequence $(U_k, V_k)$ of linked gaps of $C_1$ and $C_2$, where at least one of the diameter sequences tends to $0$.
Assume, by symmetry, that $\diam (U_k) \to 0$. Take $x_k \in \partial U_k \subseteq C_1$, and $y_k \in U_k \cap \partial V_k \subseteq C_2$. Then,
\[\dist (x_k, y_k) \leq \diam (U_k) \to 0.\]
Since $(x_k)_k \subseteq C_1$ there exists $(x_{k_j})_j$ a subsequence $(x_{k_j})_j$ convergent to $x\in C_1$. Since $(y_{k_j})_j \subseteq C_2$, we also get $(y_{k_j})_j \to x \in C_2$.
So $x \in C_1 \cap C_2$ contradicting the assumption that $C_1 \cap C_2 =\emptyset$.
\end{proof}

\section{Thickness and winning sets}

Schmidt's game and its variants are a powerful tool for investigating properties of intersections of sequences of sets, see \cite{BHNS} for a survey.
We will define a game and prove that every set with positive thickness can be seen as a winning set with certain parameters for the game. We will show that game has good properties, for example monotonicity in its parameters, invariance under similarities, and that the intersection of winning sets is again a winning set with different parameters. Theorem \ref{TEOlowerbound}, proved in the Appendix,  gives a lower bound for the Hausdorff dimension of winning sets for this game and this leads to Theorem \ref{THEOREMintersection} on the dimension of intersections.


\subsection*{Definition of the Game}
We define a game in $\R^d$ similar to the potential game from \cite{BFS} but adapted to our purposes:

\begin{definition}\label{gamedef}
Given $\alpha, \beta, \rho >0$ and $c \geq 0$, Alice and Bob play the $(\alpha, \beta, c, \rho)$-game in $\R^d$ under the following rules:
\begin{itemize}
\item For each $m \in \N_{0}$ Bob plays first, and then Alice plays.
\item On the $m$-th turn, Bob plays a closed ball $B_m:=B[x_m , \rho_m ]$, satisfying $\rho_0 \geq \rho$, and $\rho_{m}\geq \beta \rho_{m-1}$ and $B_m \subseteq B_{m-1}$ for every $m \in \N$.
\item On the $m$-th turn Alice responds by choosing and erasing a finite or countably infinite collection $\mathcal{A}_m$ of open sets. Alice's collection must satisfy $\sum_{i} (\diam A_{i,m})^c \leq (\alpha \rho_m )^c$ if $c>0$, or $\diam A_{1,m} \leq \alpha \rho_m$ if $c=0$ (in the case $c=0$ Alice can erase just one set).
\item $\lim_{m \to \infty} \rho_m =0$ (Note that this is a non-local rule for Bob. One can define the game without this rule, adding that Alice wins if $\lim_{m \to \infty} \rho_m \neq 0$. But to make the definitions simpler we added this condition as a rule for Bob.)
\end{itemize}
\end{definition}

Alice is allowed not to erase any set, or equivalently to pass her turn.

There exists a single point $x_{\infty} = \bigcap_{m \in \N_0} B_m$ called the {\it outcome of the game}. We say a set $S \subset \R^d$ is an $(\alpha, \beta, c, \rho)$-{\it winning set}, or just a {\it winning set} when the game is clear, if Alice has a strategy guaranteeing that if $x_{\infty} \notin \bigcup_{m \in \N_0} \bigcup_i A_{i,m}$, then $x_{\infty} \in S$.

Note that the conditions $B_0 \supseteq B_1 \supseteq \cdots$ and $\lim_{m \to \infty} \rho_m =0$ imply $\beta < 1$.

\subsection*{Good properties of the game}

\begin{proposition}[Countable intersection property]\label{Countable intersection property}
Let J be a countable index set, and for each $j \in J$ let $S_j$ be an $(\alpha_j , \beta, c, \rho)$-winning set, where $c>0$. Then, the set $S:= \bigcap_{j \in J} S_j$ is $(\alpha ,\beta, c, \rho)$-winning where $\alpha^c = \sum_{j \in J} \alpha_j^c$ (assuming that the series converges).
\end{proposition}
To see this, it is enough to consider the following strategy for Alice: in the turn $k$ she plays the union over $j$ of all the strategies of turn $k$.

\begin{proposition}[Monotonicity]\label{Monotonicity}
If $S$ is $(\alpha , \beta, c, \rho)$-winning and $\tilde{\alpha} \geq \alpha$, $\tilde{\beta} \geq \beta$, $\tilde{c} \geq c$ and $\tilde{\rho} \geq \rho$, then $S$ is $(\tilde{\alpha} , \tilde{\beta}, \tilde{c}, \tilde{\rho})$-winning.
\end{proposition}
This holds because  \[\Big( \sum_i \alpha_i^{\tilde{c}}\Big)^{1/\tilde{c}} \leq \Big( \sum_i \alpha_i^{c}\Big)^{1/c} \text{ when } c\leq \tilde{c},\] so Alice can answer in the $(\tilde{\alpha} , \tilde{\beta}, \tilde{c}, \tilde{\rho})$-game using her strategy to answer from the $(\alpha , \beta, c, \rho)$-game.

\begin{proposition}[Invariance under similarities]\label{Invariance under similarities}
Let $f:\R^d \to \R^d$ be a  similarity satisfying
 $$\dist (f(x),f(y))=\lambda \dist(x,y) \ \mbox{ for all }\ x,y \in \R^d.$$
Then a set $S$ is $(\alpha , \beta, c, \rho)$-winning if and only if the set $f(S)$ is $(\alpha , \beta, c, \lambda \rho)$-winning.
\end{proposition}
This holds by ``translating'' the strategies being played through $f$.

\begin{remark}[Relationship with the potential game in \cite{BFS}]
Let $\mathcal{P}$ be the set of singletons in $\R^d$.
Since every set $A$ is contained in a ball of radius $\diam(A)$, if $S \subseteq \R^d$ is an $(\alpha , \beta, c, \rho)$-winning set, then it is an $(\alpha , \beta, c, \rho, \mathcal{P})$-potential winning set in the game defined in \cite{BFS}).
\end{remark}

\subsection*{Relationship between thickness and winning sets}

We now establish the key property that relates winning sets to thickness.

\begin{proposition}\label{compact winning}
Let $C\subset \R^d$ be compact with unbounded complement $E$ and write $S:=C \cup E$. If $\tau:=\tau(C)>0$, then $S$ is $\big(\frac{1}{\tau \beta}, \beta, 0, \frac{\beta \diam (C)}{2}\big)$-winning for every $\beta \in (0,1)$.
\end{proposition}

\begin{proof}
We first describe a strategy for Alice. Given a move $B$ by Bob, how does Alice respond?
If there exists $n \in \N$ such that $B$ intersects $G_n$ and $\diam (B) < \dist(G_n, \bigcup_{1\leq i \leq n-1}G_i \cup E)$, then $B \cap G_n \neq \emptyset$ and $B \cap G_k = \emptyset$ for all $1 \leq k <n$ and $B \cap E = \emptyset$. Alice erases $G_n$ if it is a legal move, otherwise Alice does not erase anything.

To show that this strategy is winning, suppose that $x_{\infty} \notin \bigcup_m A_m$. We want to show that $x_{\infty} \in S$.
Otherwise $x_{\infty} \notin S$ so there exists $n$ such that $x_{\infty} \in G_n$. We will show that Alice erases $G_n$ at some stage of the game.
By definition $x_{\infty} \in B_m$ for all $m \in \N_0$, and we assumed $x_{\infty} \in G_n$, so $x_{\infty} \in B_m \cap G_n$ for all $m \in \N_0$.
Since $\tau>0$, then $\dist(G_n, \bigcup_{1\leq i \leq n-1}G_i \cup E)>0$. Also $\lim_{m \to \infty}\diam (B_m)=0$, so taking $m_n \in \N_0$ to be the smallest integer such that $\dist(G_n, \bigcup_{1\leq i \leq n-1}G_i \cup E) >\diam (B_{m_n})$, we know that $B_{m_n} \cap G_n \neq \emptyset$ and $B_{m_n} \cap G_k = \emptyset$ for all $1 \leq k <n$.
If $m_n=0$, then \[\diam (B_0)=2\rho_0 \geq 2 \rho=\beta \diam(C) \geq \beta \dist\Big(G_n, \bigcup_{1\leq i \leq n-1}G_i \cup E\Big).\]
If $m_n>0$, then \[\diam (B_{m_n}) \geq \beta \diam (B_{m_n-1})\geq \beta \dist\Big(G_n, \bigcup_{1\leq i \leq n-1}G_i \cup E\Big).\]
So $\diam (B_{m_n}) \geq \beta \dist\big(G_n, \bigcup_{1\leq i \leq n-1}G_i \bigcup E\big)$. Hence, $$\diam (G_n)\leq \frac{1}{\tau} \dist\Big(G_n, \bigcup_{1\leq i \leq n-1}G_i \cup E\Big) \leq \frac{1}{\tau \beta} \diam (B_{m_n})=\alpha \diam (B_{m_n}).$$
This means that it is legal for Alice to erase $G_n$ in the $m_n$-th turn, and her strategy specifies that she does so.
Finally, if $m_i=m_j$ then the first gap intersecting $B_{m_i}=B_{m_j}$ is $G_j$ and also $G_i$, so $i=j$; thus the elements of $\{m_n: \ n \in \N\}$ are all different.
\end{proof}

\begin{observation}\label{tau winning}
Let $C$ be a compact set in $\R^d$ and $\tau:= \tau(C)>0$. Then, by Proposition \ref{compact winning} and monotonicity, $S:=C \cup E$ is a $\big( \frac{1}{\tau \beta}, \beta, c, \frac{\beta}{2}\diam(C) \big)$-winning set for all $\beta \in (0,1)$ and all $c \geq 0$.
\end{observation}

\section{A lower bound for the dimension of intersections of thick compact sets in $\R^d$}\label{SectionIntersection}

Whilst the gap lemma type results, concern the intersection of just two sets, it is of interest to obtain conditions that ensure that finitely many, or even countably many compact subsets of  $\R^d$ have non-empty intersection. Using the game introduced in Definition \ref{gamedef} we not only obtain conditions involving thickness that ensure that such collection of sets in has non-empty intersection, but also get a lower bound for the Hausdorff dimension of this intersection, as stated in Theorem \ref{THEOREMintersection}.

To achieve this we use the following technical theorem that gives a lower bound for the dimension of winning sets, based on \cite[Theorem 5.5]{BFS} and \cite[Theorem 4]{AY} and proved in the Appendix A. The parameters of a winning set provide a measure of its size and we translate this in terms of thickness which is a single number that is easy to compute and work with.

\begin{theorem}\label{TEOlowerbound}
Let $S \subseteq \R^d$ be an $(\alpha, \beta, c, \rho)$-winning set with $c<d$ and $\beta \leq \frac{1}{4}$. Then for every ball $B_0$ of radius larger than $\rho$,
\[\dim_H(S\cap B_0)\ \geq d-K_1\  \frac{\alpha^{d}}{|\log(\beta)|}\ >0\ \text{ if }\ \alpha^c\leq \frac{1}{K_2}(1-\beta^{d-c}),\]
where $K_1$ and $K_2$ are as in \eqref{knumbers}.
\end{theorem}

We now prove Theorem \ref{THEOREMintersection}  by combining Theorem \ref{TEOlowerbound} with the fact that sets of positive thickness can be regarded as winning sets.

\begin{proof}[Proof of Theorem \ref{THEOREMintersection}]
By Observation \ref{tau winning}, for each $i$
\[S_i:=E_i \cup C_i \text{ is a }\bigg( \frac{1}{\tau_i \beta}, \beta, c, \frac{\beta}{2}\diam(C_i) \bigg) \text{-winning set}\]
for all $\beta \in (0,1)$ and all $c \geq 0$.
We fix $c \in (0,d)$ and $\beta \in (0,\frac{1}{4}]$ from the hypothesis (iii).
We define $\rho:=\frac{\beta}{2} \sup_i \diam (C_i)$ which is a finite number by hypothesis (i). By monotonicity, Proposition \ref{Monotonicity}, $S_i$ a $\big( \frac{1}{\tau_i \beta}, \beta, c, \rho \big)$-winning set.
Hence, by Proposition \ref{Countable intersection property}, \[S:=\bigcap_i S_i \text{ is a }(\alpha, \beta, c, \rho) \text{-winning set},\] where
 \begin{equation}\label{boundalpha}\alpha:=\Big( \sum_i (\tau_i \beta)^{-c}\Big)^{1/c}=\frac{1}{\beta} \Big( \sum_i \tau_i^{-c}\Big)^{1/c}.\end{equation}

By hypothesis (ii)  there exists a ball $B$ such that $B \cap E_i =\emptyset$ for all $i$ and we take $r$ to be the radius of $B$.
By definition of $\rho$ and $\beta$, we have $r\geq \rho$.
By hypothesis (iii) and equation \eqref{boundalpha} we have $\alpha^c \leq \frac{1}{K_2}(1-\beta^{d-c})$, hence we can apply Theorem \ref{TEOlowerbound} to get
\[\dim_H(S \cap B) \geq d- K_1 \frac{\alpha^d}{|\log (\beta)|}>0,\] and we know by definition of $\alpha$ that $d- K_1 \frac{\alpha^d}{|\log (\beta)|}=d- K_1 \frac{\left( \sum_i \tau_i^{-c}\right)^{d/c}}{\beta^d |\log (\beta)|}$.

Since $B$ does not intersect any $E_i$,
\[S_i \cap B \subseteq S_i \cap E_i^C \cap B=C_i \cap B \text{ for every }i,\]
so $S \cap B \subseteq B \cap \bigcap_i C_i$. The conclusion follows.
\end{proof}

\section{Application: Patterns in thick compact sets of $\R^d$}\label{SectionPatterns}

In this section we deduce Theorem \ref{TeoPatterns} on the existence of small copies of pattens in sufficiently thick sets from  Theorem \ref{THEOREMintersection} and illustrate this in the case of Sierpi\'{n}ski carpets.

\begin{proof}[Proof of Theorem \ref{TeoPatterns}]
We write $B_0:=\frac{1}{8}B$ for the ball with the same centre as $B$ but with radius $\frac{1}{8} \rad (B)$.
Given a finite set $A$ and $\lambda \in \big(0,\frac{\diam (B)}{16\diam (A)}\big)$ we seek translates of $\lambda A:=\{b_1, \cdots, b_n\}$ with $b_i\in \R^d$ where we can assume $b_1=0$.
As $\diam (\lambda A)< \frac{\diam (B)}{16}$ then $\lambda A \subseteq B(0, \frac{\diam (B)}{16})$.

We define $C_i:=C-b_i$ which is a compact set with thickness $\tau$ for every $1\leq i \leq n$.
By hypothesis there is a ball $B \subseteq E^C$, so there is a ball $\widetilde{B} \subseteq \bigcap_{1\leq i \leq n}(B -b_i)\subseteq \bigcap_{1\leq i \leq n}E_i^C$ of diameter $\diam (B)(1- \frac{1}{16})=\frac{15}{16}\diam (B)$.

We take $\beta:=\min \{\frac{1}{4}, \frac{\diam (\tilde{B})}{\diam (C)} \}$, $\alpha:=1/\tau \beta$ and
$c := d- 1/\log(\tau\beta)$. Then $\alpha^c=e\alpha^d$ and $d-c=1/\log(\tau \beta)$.

By Theorem \ref{THEOREMintersection}, if
\begin{equation}\label{cond cardinal}
n \alpha^c \leq \frac{1}{K_2} (1- \beta^{d-c})
\quad \text{or equivalently}\quad  n \leq \frac{1}{K_2} \alpha^{-c} (1 - \beta^{d-c})
\end{equation}
then $\dim_H (\widetilde{B} \cap \bigcap_{1\leq i \leq n} C_i)>0$.

By definition of $\alpha$, $\beta$ and $c$, and using that $f(\tau):=\log(\tau) (1 - \beta^{1/\log(\tau \beta)})$ is a decreasing function with $\lim_{\tau \to \infty}f(\tau)=|\log (\beta)|$,
\begin{align*}\frac{1}{K_2} \alpha^{-c} (1 - \beta^{d-c})
&=\frac{1}{e K_2}(\tau\beta)^d(1 - \beta^{1/\log(\tau \beta)})\\
&=\frac{1}{e K_2} \frac{\tau^{d}}{\log(\tau)} \beta^d  \log(\tau) (1 - \beta^{1/\log(\tau \beta)})\\
&\geq \frac{1}{e K_2}  \beta^d  |\log(\beta)|  \frac{\tau^{d}}{\log(\tau)}
\end{align*}
Setting
\[N(\tau):= \left\lfloor \frac{\beta^d  |\log(\beta)|}{e K_2} \frac{\tau^d}{\log(\tau)} \right\rfloor\]
 it follows from if  \eqref{cond cardinal} that if $n \leq N(\tau)$ then  $\text{dim}_{\rm H}(\widetilde{B} \cap \bigcap_{1\leq i \leq n} C_i)>0$. If $x\in X:= \widetilde{B} \cap \bigcap_{1\leq i \leq n} C_i$, then $x+b_i \in C_i +b_i=C$ for every $1\leq i \leq n$,
so $C \supseteq x+ \{ b_1, \cdots , b_n \}=x+\lambda A$ as required.
\end{proof}

\subsection*{Sierpi\'{n}ski carpets and sponges}
Sierpi\'{n}ski carpets and sponges provide examples of sets for which thickness is easily found and which satisfy Theorem \ref{TeoPatterns} .

Let $n_1, \cdots, n_d \in \N_{\geq 3}$ be odd natural numbers.
Let
$$D = \big\{{\bf i}:=(i_1,\ldots, i_d): 1\leq i_k \leq n_k, \text{ with } (i_1,\ldots, i_d)\neq  \big({\textstyle\frac{1}{2}}(n_1+1),\ldots, {\textstyle\frac{1}{2}}(n_d+1)\big)\big\}.$$
The family of affine maps
$$\big\{f_{{\bf i}}: \R^d \to \R^d: {\bf i}\in D \big\},$$
where
$$f_{{\bf i}}(x_1,\ldots,x_d) =  \Big(\frac{x_1 +i_1-1}{n_1}, \ldots, \frac{x_d +i_d-1}{n_d}\Big),$$
forms an iterated function system, which defines a unique non-empty compact set $C\subset \R^d$ such that $C = \bigcup_{{\bf i}\in D} f_{{\bf i}}(C)$,   see \cite{FalconerBook}. Then $C$ is a self-affine Sierpi\'{n}ski sponge (carpet if $d=2$) which can also be realised iteratively by repeatedly substituting the coordinate parallelepipeds obtained by dividing the unit cube $[0,1]^d$ into $n_1\times\cdots\times n_d$ smaller parallelepipeds, with the central one removed, into themselves. In other words
$$C= \bigcap_{k=0}^\infty\  \bigcup_{{\bf i}_1, \ldots , {\bf i}_k \in D}  f_{{\bf i}_1}\circ\cdots\circ f_{{\bf i}_k}([0,1]^d).$$

We will find the  thickness of $C$.
Each parallelepiped at the $k$th step of the iterative construction has side-lengths $1/{n_i^k}$ ($1\leq i \leq d$). Thus the central parallelepipeds that are removed and which form  gaps at the  $k$th step have diameter
\[\diam_k:=\sqrt{\sum_{1\leq k \leq d}\frac{1}{n_i^{2k}}}.\]
The  minimal distance of a gap removed  at the $k$th step from the previous gaps and the external complementary component $E$ is
\[\dist_k:= \min_{1\leq i \leq d}\frac{1}{n_i^k}\frac{n_i -1}{2}.\]
Hence, the thickness of $C$ is
\[\tau:=\tau(C)=\inf_{k \in \N}\frac{\dist_k}{\diam_k}=\inf_{k \in \N}\frac{\min_{1\leq i \leq d}\frac{1}{n_i^k}\frac{n_i -1}{2}}{\sqrt{\sum_{1\leq k \leq d}\frac{1}{n_i^{2k}}}}.\]

Thus, with $\beta:=\min\{\frac{1}{4}, \frac{15}{16 \sqrt{d}}\}$, Theorem \ref{TeoPatterns} gives that $C$ contains homothetic copies of every pattern with at most $N(\tau)$ points where $N(\tau)$ is given by \eqref{ntau}.

For example,  the self-similar Sierpi\'{n}ski carpet $C_n$ in $\R^2$, taking $n_1 = n_2 =n$ above, has thickness $\tau = (n-1)\big/2\sqrt{2}$, so there is a homothetic copy in $C_n$ of every configuration of  up to $N(\tau)$ points.
Because $K_2$ is large, $n$ needs to be large to guarantee even that similar copies of all triangles can be found in $C_n$. On the other hand, for  $C_n$ to contain copies of all $k$-point configurations, $n = O((k\log k)^{1/2}) $ which does not increase too rapidly for large $k$.


\section{Thickness and Hausdorff dimension}\label{DimTau}
In this section we obtain two different lower bounds for the Hausdorff dimension of sets in $\R^d$ in terms of their thickness.

Firstly,  Theorem \ref{THEOREMintersection} yields a lower bound by taking a single set $C$.

\begin{corollary}\label{CoroLowerBound}
Let $C$ be a compact set in $\R^d$ with positive thickness $\tau$ (so $\diam (C) <+\infty$ and there is a ball $B$ such that $B \cap E=\emptyset$). If there exists $c \in (0,d)$ such that $\tau^{-c}\leq \frac{1}{K_2}\beta^c (1-\beta^{d-c})$ for $\beta:=\min \{\frac{1}{4}, \frac{\diam (B)}{\diam(C)}\}$
then
\[\dim_H(C)\geq \dim_H (B \cap C) \geq d- K_1 \frac{\tau^{-d}}{\beta^d |\log (\beta)|}>0.\]
\end{corollary}

Secondly, we can get a lower bound in the case of convex sets with convex gaps by considering 1-dimensional sections.

\begin{proposition}\label{PropLineInt}
Let $C_0$ be a proper compact convex set in $\R^d$ where $d\geq 2$, and let $C = C_0 \setminus \bigcup_{k=1}^\infty G_k$, where $\{G_k\}_k$ are open convex gaps ordered by decreasing diameters. Then   $\tau(C\cap L)\geq \tau(C)$  for every straight line $L$ that properly intersects $C_0$.
\end{proposition}

\begin{proof}
Let $L$ be a straight line that properly intersects $C_0$. Let $\{I_i\}_{i=1}^\infty$ be the (countable or finite) set of open intervals  $I_i := G_{k(i)}\cap L$ in $L$ ordered so that $|I_i| \leq |I_j|$ if $i\geq j$, where $|\ |$ denotes the length of an interval. Let  $1\leq i \leq j-1$. There are two cases:

\noindent\quad (a)  if $k(i)<k(j)$ then
$$\mbox{\rm dist}(I_i,I_j)\  \geq\  \mbox{\rm dist}(G_{k(i)},G_{k(j)})\ \geq\  \tau(C)\diam (G_{k(j)})\  \geq\  \tau(C)|I_j|;$$
\quad(b) if $k(j)<k(i)$ then
$$\mbox{\rm dist}(I_i,I_j)\  \geq\  \mbox{\rm dist}(G_{k(i)},G_{k(j)})\ \geq\  \tau(C)\diam (G_{k(i)})\  \geq\  \tau(C)|I_i|\ \geq\  \tau(C)|I_j|.$$
In both cases $\mbox{\rm dist}(I_i,I_j)\geq |I_j|$ for all  $1\leq i \leq j-1$ so $\tau(C\cap L)\geq \tau(C)$ from the definition of thickness.
\end{proof}

We can now obtain a lower bound for the Hausdorff dimension for these sets in terms of thickness using the bound \eqref{TauDimBook} fore sets in $\R$.

\begin{proposition}\label{PropLowerBound}
Let $C_0 \subseteq \R^d$ be a proper compact convex set, and let $C = C_0 \setminus \bigcup_{i=1}^\infty G_i$ where $\{G_i\}_i$ are open convex gaps.
Then
\begin{equation}\label{dimhigher}
\dim_H (C)\ \geq\  d-1 + \frac{\log 2}{\log(2+1/\tau(C))}
\end{equation}
where $\tau(C)$ is the thickness of $C$.
\end{proposition}

\begin{proof}
Let $L$ be a straight line that properly intersects $C_0$.  Combining the relationship between thickness and Hausdorff dimension for  subsets of $\R$  stated in \eqref{TauDimBook} with Proposition \ref{PropLineInt},
$$\dim_H (C\cap L )\geq \frac{\log 2}{\log(2+1/\tau (C\cap L))}\geq \frac{\log 2}{\log(2+1/\tau (C))}.$$
This is true for all lines $L$ in a given direction that properly intersect $C_0$, so by a standard result relating the Hausdorff dimension of a set to the Hausdorff dimensions of parallel sections, see for example, \cite[Corollary 7.10]{FalconerBook}, inequality \eqref{dimhigher} follows.
\end{proof}


\begin{observation}
When $d=1$ Proposition \ref{PropLowerBound}
is better than Corollary \ref{CoroLowerBound}.
For $d\geq 2$,  Corollary \ref{CoroLowerBound} gives a better bound than Proposition \ref{PropLowerBound} when $\tau$ is large but when $\tau$ is small Proposition \ref{PropLowerBound} is better.
\end{observation}


\appendix

\section{Proof of Theorem \ref{TEOlowerbound}}

The proof of Theorem \ref{TEOlowerbound} is based on \cite[Theorem 5.5]{BFS} and \cite[Theorem 4]{AY} and adapted to our particular setting.

\begin{proof}
We can assume without loss of generality that the radius of $B_0$ is $\rho$. We let $x_0$ be the center of $B_0$, \[\rho_n:=\beta^n \rho \text{ radii of balls, and}\] \[E_n:=\frac{\rho_n}{2}\Z^d +x_0 \text{ centers of balls of the family}\] \[\mathcal{E}_n:=\left\{ B(\frac{\rho_n}{2}z +x_0, \rho_n): \ z \in \Z^d \right\}.\]
We will take Bob's move of the $n$-turn from $\mathcal{E}_n$.

We also define \[D_n:=3 \rho_n \Z^d+x_0 \subset E_n,\] \[\mathcal{D}_n:=\{B(3 \rho_n z+x_0, \rho_n): \ z \in \Z^d \} \subset \mathcal{E}_n.\] Note that the elements of $\mathcal{D}_n$ are disjoint (moreover they are at distance $\rho_n$).

We fix $\gamma \in (0,1)$, a small number to be determined later (independent of $\alpha$, $\beta$, $c$ and $\rho$). Let $N:=\lfloor \frac{\gamma^d}{\alpha^d}\rfloor$.

We define the function $\pi_n : \mathcal{E}_{n+1} \to \mathcal{E}_n$, $B \mapsto \pi_n(B)$ in the following way:
\begin{itemize}
\item When $n\neq jN$ for all $j$: we define $\pi_n(B)$ as the element of $\mathcal{E}_n$ that contains $B$ such that $B$ is as centered as possible inside that element.
\item When $n=jN$ for some $j$: If there exists $B'\in \mathcal{D}_{jN}$ containing $B$, we define $\pi_n(B):=B'$ (it is well defined because in that case there is only one element belonging to $\mathcal{D}_{jN}$). If not, we define the function as before.
\end{itemize}

Intuitively the function $\pi_n$ carries the elements of level $n+1$ to its ancestor of level $n$.

We use the following notation: for $m<n$ and $B \in \mathcal{E}_n$, $\pi_m(B):= \pi_m \circ \pi_{m+1} \circ \cdots \circ \pi_{n-1}(B) \in \mathcal{E}_m$. This is to say, we carry $B$ to its ancestor of level $m$ via the functions $\pi$.
If Bob plays $B \in \mathcal{E}_n$ in the turn $n$, we consider that in the previous turns $m \in \{0, \cdots, n-1\}$ Bob has played $\pi_m(B)$.
Then, we have the following inclusions of movements from the turn $n$ to the turn $0$: \[B \subset \pi_{n-1}(B) \subset \cdots \subset \pi_0(B).\]
We defined the function in this way to guarantee that Bob's moves are legal.
Alice responds under her winning strategy. If in the turn $n$ Bob plays $B \in \mathcal{E}_n$, we define $\mathcal{A}(B)$ as Alice's answer (each $A \in \mathcal{A}(B)$ is a countable collection of sets $A:= \{A_{i,n}\}_i$, and a legal movement as an answer for $B$, i.e.: $\sum_i \diam (A_{i,n})^c \leq (\alpha \rho_n)^c$). Let \[\mathcal{A}^*_m(B):=\{A \in \mathcal{A}(\pi_m(B)) : B\cap A \neq \emptyset \}\] be Alice's answer (this is a list of sets) to the ancestor of $B$ of level $m<n$.

Given any ball $B$, we denote by $\frac{1}{2}B$ the ball with the same center as $B$ and the half of the radius.

Note that as $\beta \leq \frac{1}{4}$, if $B \in \mathcal{D}_{jN}$ and $B'\in \mathcal{E}_{jN+1}$ satisfying that $B' \cap \frac{1}{2}B \neq \emptyset$, then $B' \subset B$, so $\pi_{jN}(B')=B$. It follows that \begin{equation}\label{eqpi}\text{if } n>jN, \ B' \subset \frac{1}{2}B \text{ with } B'\in \mathcal{E}_n \text{ and } B \in \mathcal{D}_{jN}, \text{ then } \pi_{jN}(B')=B.\end{equation}
This is true because if we look at the ancestor of $B'$ of level $jN+1$, since $\pi_n$ chooses the element belonging to $\mathcal{E}_n$ that contains $B$ such that $B$ is as centered as possible, that element must intersect $\frac{1}{2}B$.

We define for every $B \in \mathcal{D}_j$ \[\phi_j (B):= \sum_{n<j} \ \sum_{A \in \mathcal{A}^*_n(B)} \diam(A_{i,n})^c.\] This is a measure of all of Alice's answers to the ancestors of $B$. Note that $\phi_0 (B)=0$.

Let
\[\mathcal{D}'_j:=\{ B \in \mathcal{D}_j: \ \phi_j(B) \leq (\gamma \rho_j)^c\}.\]
We define \[\mathcal{D}_j(B):=\{B' \in \mathcal{D}_j : \ B' \subset \frac{1}{2}B\}.\]

\subsubsection*{Some useful bounds}
\begin{observation}\label{cota card 1}
If $B\in \mathcal{D}'_{jN}$, we have that $\text{rad} (\frac{1}{2}B)=\frac{1}{2}\beta^{jN}\rho$, and $\text{rad}(B')=\beta^{(j+1)N}\rho$ for every $B' \in \mathcal{D}_{(j+1)N}$.
We can cover $\frac{1}{2}B$ with enlarged balls from $\mathcal{D}_{(j+1)N}(B)$ (with radii $4 \rho_{(j+1)N} \sqrt{d}$). This gives us a lower bound for  $\#\mathcal{D}_{(j+1)N}(B)$:
\[\mathcal{L}^d(\frac{1}{2}B) \leq \#\mathcal{D}_{(j+1)N}(B) \mathcal{L}^d(B_{4 \rho_{(j+1)N} \sqrt{d}}),\] so
\[\beta^{-Nd} \frac{1}{2^d 4^d \sqrt{d}^d} \leq \#\mathcal{D}_{(j+1)N}(B).\]
\end{observation}

\begin{proposition}\label{cota card 2}\label{canthijos}
If $\alpha^c \leq \frac{1}{K_2}(1-\beta^{d-c})$ where $K_2 :=\max\{\gamma^{-2d}, 2 \gamma^{-d} \log (\gamma^{-d})\}$, we have that \[\#(\mathcal{D}_{(j+1)N}(B)\cap \mathcal{D}'_{(j+1)N}) \geq \beta^{-Nd} \left( \frac{1}{2^d 4^d \sqrt{d}^d} -3^d\gamma^d (1+4^d 2) \right) \text{ for all } B\in \mathcal{D}'_{jN}.\]
\end{proposition}

We denote by $\text{rad}(B)$ the radius of the ball $B$. We start by proving two preliminary lemmas:
\begin{lemma}\label{cota card 3}
\begin{enumerate}[a)]
\item For all $n \in \N$ and $B' \in \mathcal{E}_n$ we have that
\[\sum_{A \in \mathcal{A}(B')} \min\left\{1, \frac{\diam (A_{i,n})^c}{(\gamma \rho_{(j+1)N})^c}\right\}\left( \frac{\diam (A_{i,n})+2 \rho_{(j+1)N}}{\textup{rad}(B')}\right)^d \leq 3^d \alpha^c \max\left\{\alpha^{d-c}, \gamma^{-c} \left( \frac{\rho_{(j+1)N}}{\textup{rad}(B')} \right)^{d-c}\right\}.\]
\item If $B\in \mathcal{D}'_{jN}$ then \[\sum_{n<jN} \ \sum_{A \in \mathcal{A}^*_n(B)} \min\left\{1, \frac{\diam (A_{i,n})^c}{(\gamma \rho_{(j+1)N})^c}\right\} \left(\frac{\diam (A_{i,n})+2 \rho_{(j+1)N}}{\textup{rad}(B)}\right)^d \leq 3^d \gamma^c \max\left\{\gamma^{d-c}, \gamma^{-c} \left( \frac{\rho_{(j+1)N}}{\textup{rad}(B)} \right)^{d-c}\right\}.\]
\end{enumerate}
\end{lemma}

\begin{proof}[Proof of Lemma \ref{cota card 3}]

Firstly, splitting into the cases $x\leq y$ and $y \leq x$, it is easy to see that \begin{equation}\label{cotaxy}\min\left\{1,\frac{x^c}{(\gamma y)^c}\right\}(x+2y)^d \leq 3^d x^c \max\left\{x^{d-c}, \frac{y^{d-c}}{\gamma^c}\right\} \text{ for all }x,y>0.\end{equation}

Secondly, we will prove that if $n \in \N$ and $B' \in \mathcal{E}_n$ then the claim a) holds.
By applying the inequality \eqref{cotaxy} to $x:=\frac{\diam (A_{i,n})}{\text{rad}(B')}$ and $y:=\frac{\rho_{(j+1)N}}{\text{rad}(B')}$, summing over all $A_{i,n} \in \mathcal{A}(B')$ and using that Alice is playing legally, we have that
\begin{align*}
&\sum_{A_{i,n} \in \mathcal{A}(B')} \min\left\{1, \frac{\diam (A_{i,n})^c}{(\gamma \rho_{(j+1)N})^c}\right\} \left( \frac{\diam (A_{i,n})+2 \rho_{(j+1)N}}{\text{rad}(B')}\right)^d\\
&\leq 3^d  \sum_{A_{i,n} \in \mathcal{A}(B')} \left(\frac{\diam (A_{i,n})}{\text{rad}(B')} \right)^c \max\left\{\left(\frac{\diam (A_{i,n})}{\text{rad}(B')}\right)^{d-c}, \gamma^{-c} \left(\frac{\rho_{(j+1)N}}{\text{rad}(B')}\right)^{d-c} \right\}\\
&\leq 3^d \max\left\{ \left(\max_{A_{i,n} \in \mathcal{A}(B')} \frac{\diam (A_{i,n})}{\text{rad}(B')}\right)^{d-c}, \gamma^{-c} \left(\frac{\rho_{(j+1)N}}{\text{rad}(B')}\right)^{d-c} \right\} \sum_{A_{i,n} \in \mathcal{A}(B')} \left(\frac{\diam (A_{i,n})}{\text{rad}(B')} \right)^c\\
&\leq 3^d \alpha^c \max\left\{\alpha^{d-c}, \gamma^{-c} \left(\frac{\rho_{(j+1)N}}{\text{rad}(B')}\right)^{d-c}\right\}.
\end{align*}

Finally, we prove the claim b). By applying the inequality \eqref{cotaxy} to $x:=\frac{\diam (A_{i,n})}{\text{rad}(B)}$ and $y:=\frac{\rho_{(j+1)N}}{\text{rad}(B)}$, summing over all elements of $\bigcup_{n <jN}\mathcal{A}^*_n(B)$, and using that, since $B\in \mathcal{D}'_{jN}$, we have \[\sum_{n<jN} \ \sum_{A_{i,n} \in \mathcal{A}^*_n(B)}\left(\frac{\diam (A_{i,n})}{\text{rad}(B)}\right)^c \leq \gamma^c,\] and in particular $\frac{\diam (A_{i,n})}{\text{rad}(B)} \leq \gamma$ for every $i$ and every $n<jN$, we obtain that:
\begin{align*}
&\sum_{n<jN} \ \sum_{A_{i,n} \in \mathcal{A}^*_n(B)} \min\left\{1, \frac{\diam (A_{i,n})^c}{(\gamma \rho_{(j+1)N})^c}\right\}\left( \frac{\diam (A_{i,n})+2 \rho_{(j+1)N}}{\text{rad}(B)} \right)^d\\
&\leq 3^d \sum_{n<jN} \ \sum_{A_{i,n} \in \mathcal{A}^*_n(B)}\left(\frac{\diam (A_{i,n})}{\text{rad}(B)} \right)^c \max\left\{\left(\frac{\diam (A_{i,n})}{\text{rad}(B)} \right)^{d-c}, \gamma^{-c} \left(\frac{\rho_{(j+1)N}}{\text{rad}(B)}\right)^{d-c}\right\}\\
&\leq 3^d \left( \sum_{n<jN} \ \sum_{A_{i,n} \in \mathcal{A}^*_n(B)}\left(\frac{\diam (A_{i,n})}{\text{rad}(B)} \right)^c\right) \max\left\{\gamma^{d-c}, \gamma^{-c} \left(\frac{\rho_{(j+1)N}}{\text{rad}(B)}\right)^{d-c}\right\}\\
&\leq 3^d \gamma^c \max\left\{\gamma^{d-c}, \gamma^{-c} \left( \frac{\rho_{(j+1)N}}{\text{rad}(B)} \right)^{d-c}\right\}.
\end{align*}

\end{proof}

Now we are ready to prove Proposition \ref{cota card 2}.

\begin{proof}[Proof of Proposition \ref{cota card 2}]

\begin{align}\label{cotasetminus1}
\#(\mathcal{D}_{(j+1)N}(B)\setminus \mathcal{D}'_{(j+1)N})&\leq\# \left\{ B' \in \mathcal{D}_{(j+1)N}(B) : \ \frac{\phi_{(j+1)N}(B')}{(\gamma \rho_{(j+1)N})^c} >1 \right\}\nonumber\\
&\leq \sum_{B' \in \mathcal{D}_{(j+1)N}(B)} \min\left\{1, \frac{\phi_{(j+1)N}(B')}{(\gamma \rho_{(j+1)N})^c}\right\}\nonumber\\
&\leq \sum_{B' \in \mathcal{D}_{(j+1)N}(B)} \min\left\{1, \sum_{n<(j+1)N}\sum_{A_{i,n} \in \mathcal{A}^*_n(B')} \frac{\diam (A_{i,n})^c}{(\gamma \rho_{(j+1)N})^c}\right\}\nonumber\\
&\leq \sum_{B' \in \mathcal{D}_{(j+1)N}(B)} \sum_{n<(j+1)N} \sum_{A_{i,n} \in \mathcal{A}^*_n(B')}  \min\left\{1,\frac{\diam (A_{i,n})^c}{(\gamma \rho_{(j+1)N})^c}\right\} \nonumber\\
&\leq \sum_{B' \in \mathcal{D}_{(j+1)N}(B)} \sum_{n<jN} \sum_{A_{i,n} \in \mathcal{A}^*_n(B')}  \min\left\{1,\frac{\diam (A_{i,n})^c}{(\gamma \rho_{(j+1)N})^c}\right\}\nonumber\\
& \hspace{0.5cm} +\sum_{B' \in \mathcal{D}_{(j+1)N}(B)} \sum_{jN \leq n<(j+1)N} \ \sum_{A_{i,n} \in \mathcal{A}^*_n(B')}  \min\left\{1,\frac{\diam (A_{i,n})^c}{(\gamma \rho_{(j+1)N})^c}\right\}.
\end{align}

We have split the sum in \eqref{cotasetminus1} into two parts, depending on whether $n<jN$ or $jN \leq n <(j+1)N$.

To get a bound for the left-hand sum of \eqref{cotasetminus1} we will use that if $n<jN$ then
\begin{align*}
&\left\{(B', A):\ B' \in \mathcal{D}_{(j+1)N}(B), A\in \mathcal{A}^*_n(B') \right\} \\
&\subset \left\{(B', A):\ B' \in \mathcal{D}_{(j+1)N}(B), A\in \mathcal{A}^*_n(B), A\cap B' \neq \emptyset \right\}.
\end{align*}
Since $B\in \mathcal{D}'_{jN}$ the set $\mathcal{A}^*_{n}(B)$ only makes sense for $n<jN$.
This inclusion holds because of Equation \eqref{eqpi}, as $\mathcal{A}(\pi_n(B')) \subset \mathcal{A}(\pi_n(B))$ since $B' \subset B$.

So,
\begin{align*}
&\sum_{B' \in \mathcal{D}_{(j+1)N}(B)} \sum_{A_{i,n} \in \mathcal{A}^*_n(B')} \min\left\{1,\frac{\diam (A_{i,n})^c}{(\gamma \rho_{(j+1)N})^c}\right\}\\
&\leq  \sum_{A_{i,n} \in \mathcal{A}^*_n(B)} \sum_{B' \in \mathcal{D}_{(j+1)N}(B) \atop B'\cap A\neq \emptyset} \min\left\{1,\frac{\diam (A_{i,n})^c}{(\gamma \rho_{(j+1)N})^c}\right\}\\
&= \sum_{A_{i,n} \in \mathcal{A}^*_n(B)} \min\left\{1,\frac{\diam (A_{i,n})^c}{(\gamma \rho_{(j+1)N})^c}\right\} \#\{B' \in \mathcal{D}_{(j+1)N}(B), B'\cap A\neq \emptyset\}.
\end{align*}

Now, we will get a bound for the right-hand sum in \eqref{cotasetminus1}, when $jN \leq n < (j+1)N$. Recall that $B \in \mathcal{D}'_{jN}$.
First, note that if $B''\in \mathcal{D}_{(j+1)N}(B)$, then $B'' \in \mathcal{E}_{(j+1)N}$, $B'' \subset \frac{1}{2}B$ where $B\in \mathcal{D}_{jN}$. If we take $B':=\pi_{jN}(B'') \in \mathcal{E}_{jN}$, by  \eqref{eqpi} we have that $B'=B$.

For all $B'' \in \mathcal{D}_{(j+1)N}(B)$, there exists $B' \in \mathcal{E}_n$ with $B' \subset B$ and $B''\subset \frac{1}{2}B'$ ($B'=B$ if $n=jN$). Hence
\begin{align*}
&\sum_{B'' \in \mathcal{D}_{(j+1)N}(B)}  \sum_{A_{i,n} \in \mathcal{A}^*_n(B'')} \min\left\{1,\frac{\diam (A_{i,n})^c}{(\gamma \rho_{(j+1)N})^c}\right\}\\
&\leq \sum_{B' \in \mathcal{E}_n \atop B'\subset B}   \sum_{B'' \in \mathcal{D}_{(j+1)N}(B')}   \sum_{A \in \mathcal{A}(B') \atop A\cap B''\neq \emptyset} \min\left\{1,\frac{\diam (A_{i,n})^c}{(\gamma \rho_{(j+1)N})^c}\right\}\\
&= \sum_{B' \in \mathcal{E}_n \atop B'\subset B}  \sum_{A \in \mathcal{A}(B')}  \sum_{B'' \in \mathcal{D}_{(j+1)N}(B') \atop A\cap B''\neq \emptyset} \min\left\{1,\frac{\diam (A_{i,n})^c}{(\gamma \rho_{(j+1)N})^c}\right\}\\
&=\sum_{B' \in \mathcal{E}_n \atop B'\subset B}  \sum_{A \in \mathcal{A}(B')} \min\left\{1,\frac{\diam (A_{i,n})^c}{(\gamma \rho_{(j+1)N})^c}\right\} \#\{B'' \in \mathcal{D}_{(j+1)N}(B'): \ A\cap B''\neq \emptyset\},
\end{align*}
where the inequality holds by considering in particular $B':=\pi_n(B'')\subset B$.

By inequality \eqref{cotasetminus1}, and what we have noted before,

\begin{align}
&\#(\mathcal{D}_{(j+1)N}(B)\setminus \mathcal{D}'_{(j+1)N})\nonumber\\
&\leq \sum_{B' \in \mathcal{D}_{(j+1)N}(B)}  \sum_{n<jN}  \sum_{A_{i,n} \in \mathcal{A}^*_n(B')}  \min\left\{1,\frac{\diam (A_{i,n})^c}{(\gamma \rho_{(j+1)N})^c}\right\}\nonumber\\
&\quad +\sum_{B' \in \mathcal{D}_{(j+1)N}(B)}  \sum_{n=jN}^{(j+1)N-1}  \sum_{A_{i,n} \in \mathcal{A}^*_n(B')}  \min\left\{1,\frac{\diam (A_{i,n})^c}{(\gamma \rho_{(j+1)N})^c}\right\}\nonumber\\
&\leq  \sum_{n<jN}  \sum_{A_{i,n} \in \mathcal{A}^*_n(B)}  \min\left\{1,\frac{\diam (A_{i,n})^c}{(\gamma \rho_{(j+1)N})^c}\right\} \#\{B' \in \mathcal{D}_{(j+1)N}(B): \ B'\cap A \neq \emptyset\} \nonumber\\
&+\sum_{n=jN}^{(j+1)N-1}  \sum_{B' \in \mathcal{E}_n \atop B'\subset B}  \sum_{A_{i,n} \in \mathcal{A}^*_n(B')}  \min\left\{1,\frac{\diam (A_{i,n})^c}{(\gamma \rho_{(j+1)N})^c}\right\} \#\{B'' \in \mathcal{D}_{(j+1)N}(B'): \ B''\cap A \neq \emptyset\}\nonumber\\
&\leq  \left( \frac{\text{rad}(B)}{\rho_{(j+1)N}} \right)^d \sum_{n<jN}\sum_{A_{i,n} \in \mathcal{A}^*_n(B)}  \min\left\{1,\frac{\diam (A_{i,n})^c}{(\gamma \rho_{(j+1)N})^c}\right\} \left( \frac{\diam (A_{i,n})+2\rho_{(j+1)N}}{\text{rad}(B)} \right)^d \nonumber\\
&\quad+\sum_{n=jN}^{(j+1)N-1}  \sum_{B' \in \mathcal{E}_n \atop B'\subset B} \left( \frac{\text{rad}(B')}{\rho_{(j+1)N}} \right)^d \sum_{A_{i,n} \in \mathcal{A}^*_n(B')}  \min\left\{1,\frac{\diam (A_{i,n})^c}{(\gamma \rho_{(j+1)N})^c}\right\} \left( \frac{\diam (A_{i,n})+2\rho_{(j+1)N}}{\text{rad}(B')}\right)^d,
\label{cotasetminus2}
\end{align}
where in the first term of the last inequality we use that $B \in \mathcal{D}'_{jN}$ (so $\phi_{jN}(B)\leq (\gamma \rho_{jN})^c$), in the second term that Alice is playing legally (i.e.: $\sum_i \diam(A_{i,m})^c \leq (\alpha \rho_m)^c$), and in both terms that:
for every $B'\in \bigcup_n \mathcal{E}_n$ and every $A_{i,n}$, since the elements of $\mathcal{D}_{(j+1)N}$ are disjoint, $\mathcal{L}^d(B'')=C_d \rho_{(j+1)N}^d$, and if moreover $B'' \cap A_{i,n} \neq \emptyset$ then $B'' \subset \mathcal{N}(A_{i,n}, 2\rho_{(j+1)N})$ (the $2\rho_{(j+1)N}$-neighborhood of $A_{i,n}$), which is contained in a ball of radius $\diam (A_{i,n})+2 \rho_{(j+1)N}$. Therefore,
\begin{align*}
&\#\{B''\in \mathcal{D}_{(j+1)N}(B'): \ B''\cap A_{i,n} \neq \emptyset\}C_d \rho_{(j+1)N}^d
=\mathcal{L}^d\bigg( \bigcup_{B''\in \mathcal{D}_{(j+1)N}(B') \atop B''\cap A_{i,n} \neq \emptyset} B''\bigg)\\
&\leq \mathcal{L}^d\left(\mathcal{N}(A_{i,n}, 2\rho_{(j+1)N}) \right)
\leq C_d (\diam (A_{i,n})+2 \rho_{(j+1)N})^d,
\end{align*}
in other words,
 \[\#\{B''\in \mathcal{D}_{(j+1)N}(B'): \ B''\cap A_{i,n} \neq \emptyset\} \leq \frac{(\diam (A_{i,n})+2 \rho_{(j+1)N})^d}{\rho_{(j+1)N}^d}.\]

By inequality \eqref{cotasetminus2},  using claim b) from Lemma \ref{cota card 3} to bound the first term, and claim a) from Lemma \ref{cota card 3} to bound the second one, we obtain:
\begin{align}\label{cotasetminus3}
&\#(\mathcal{D}_{(j+1)N}(B)\setminus \mathcal{D}'_{(j+1)N})\nonumber\\
&\leq  \left( \frac{\text{rad}(B)}{\rho_{(j+1)N}} \right)^d \sum_{n<jN}  \sum_{A_{i,n} \in \mathcal{A}^*_n(B)}  \min\left\{1,\frac{\diam (A_{i,n})^c}{(\gamma \rho_{(j+1)N})^c}\right\} \left( \frac{\diam (A_{i,n})+2\rho_{(j+1)N}}{\text{rad}(B)}\right)^d \nonumber\\
&+ \sum_{jN \leq n<(j+1)N}  \sum_{B' \in \mathcal{E}_n \atop B'\subset B} \left( \frac{\text{rad}(B')}{\rho_{(j+1)N}}\right)^d \ \sum_{A_{i,n} \in \mathcal{A}^*_n(B')}  \min\left\{1,\frac{\diam (A_{i,n})^c}{(\gamma \rho_{(j+1)N})^c}\right\} \left( \frac{\diam (A_{i,n})+2\rho_{(j+1)N}}{\text{rad}(B')}\right)^d \nonumber\\
&\leq \left( \frac{\text{rad}(B)}{\rho_{(j+1)N}}\right)^d 3^d \gamma^c \max\left\{\gamma^{d-c}, \gamma^{-c} \left( \frac{\rho_{(j+1)N}}{\text{rad}(B)} \right)^{d-c}\right\}\nonumber\\
& + \sum_{jN \leq n<(j+1)N}  \sum_{B' \in \mathcal{E}_n \atop B'\subset B} \left(\frac{\text{rad}(B')}{\rho_{(j+1)N}} \right)^d 3^d \alpha^c \max\left\{\alpha^{d-c}, \gamma^{-c} (\frac{\rho_{(j+1)N}}{\text{rad}(B')})^{d-c}\right\}
\end{align}
To continue the estimates, we will use that \[\frac{\text{rad}(B)}{\rho_{(j+1)N}}=\frac{\beta^{jN}\rho}{\beta^{(j+1)N}\rho}=\beta^{-N}.\]
To bound the second term in \eqref{cotasetminus3}  we write $n=(j+1)N-k$ for some $k \in \{1, \cdots, N\}$. We know that $B:=B(3\rho_{jN}z+x_0, \rho_{jN})$ for some $z \in \Z^d$, and recall that $\mathcal{E}_n:=\{B(\frac{\rho_n}{2}z'+x_0, \rho_n): \ z' \in \Z^d \}$.
So,
\[\frac{\text{rad}(B')}{\rho_{(j+1)N}}=\beta^{-k}\] and
\begin{align*}
\#\{B' \in \mathcal{E}_n : \ B'\subset B \}&\leq \#\{\frac{\rho_n}{2}z'+x_0 \in B: \ z' \in \Z^d \}\\
&=\#\left\{z' \in \Z^d \cap B\left(\frac{6}{\beta^{N-k}}z, 2\Big(\frac{1}{\beta^{N-k}}-1\Big)\right)\right\}\\
&\leq \left( 4 \Big(\frac{1}{\beta^{N-k}}-1\Big)+1 \right)^d \leq 4^d \frac{1}{\beta^{d(N-k)}}.
\end{align*}
Combining with \eqref{cotasetminus3},
\begin{align}\label{cotasetminus4}
&\#(\mathcal{D}_{(j+1)N}(B)\setminus \mathcal{D}'_{(j+1)N})\nonumber\\
&\leq \left(\frac{\text{rad}(B)}{\rho_{(j+1)N}}\right)^d 3^d \gamma^c \max\Big\{\gamma^{d-c}, \gamma^{-c} \Big( \frac{\rho_{(j+1)N}}{\text{rad}(B)} \Big)^{d-c}\Big\}\nonumber\\
& \hspace{0.5cm} + \sum_{jN \leq n<(j+1)N} \ \sum_{B' \in \mathcal{E}_n \atop B'\subset B} \left(\frac{\text{rad}(B')}{\rho_{(j+1)N}}\right)^d 3^d \alpha^c \max\Big\{\alpha^{d-c}, \gamma^{-c} \Big(\frac{\rho_{(j+1)N}}{\text{rad}(B')}\Big)^{d-c}\Big\}\nonumber\\
&\leq \beta^{-Nd} 3^d \gamma^c \max\{\gamma^{d-c}, \gamma^{-c} \beta^{N(d-c)}\} + \sum_{1 \leq k \leq N} 4^d \beta^{-d(N-k)} \beta^{dk} 3^d \alpha^c \max\{\alpha^{d-c}, \gamma^{-c} \beta^{k(d-c)}\}\nonumber\\
&\leq \beta^{-Nd} 3^d \bigg( \max\{\gamma^d, \beta^{N(d-c)}\} + 4^d \Big( N \alpha^d + \alpha^c \gamma^{-c} \sum_{1 \leq k \leq N} \beta^{k(d-c)} \Big) \bigg),\nonumber\\
\end{align}
where in the last inequality we have used that if $a_n,b_n \geq 0$ then $\sum_n \max\{a_n, b_n\} \leq \sum_n a_n + \sum_n b_n$.

If we could show the following claims:
\begin{enumerate}
\item $N \alpha^d \leq \gamma^d$,
\item $\alpha^c \gamma^{-c} \sum_{k \in N_0} \beta^{k(d-c)}\leq \gamma^d$,
\item $\beta^{N(d-c)} \leq \gamma^d$,
\end{enumerate}
then \[\#(\mathcal{D}_{(j+1)N}(B)\setminus \mathcal{D}'_{(j+1)N}) \leq \beta^{-Nd}3^d\gamma^d (1+4^d 2);\]
hence, by Observation \ref{cota card 1},
\[\#(\mathcal{D}_{(j+1)N}(B)\cap \mathcal{D}'_{(j+1)N}) \geq \beta^{-Nd} \bigg(\frac{1}{2^d 4^d \sqrt{d}^d} -3^d\gamma^d (1+4^d 2) \bigg),\]
as required.

Let us prove (1)--(3):
\begin{enumerate}
\item This holds by the definition of $N$.
\item Take $K_2 :=\max\{\gamma^{-2d}, 2 \gamma^{-d} \log (\gamma^{-d})\}$.
By hypothesis and by using  $c \in (0,d)$, $\beta \in (0,\frac{1}{4}]$, we have
\[\frac{\alpha^c}{\gamma^c(1-\beta^{d-c})} \leq \frac{1}{\gamma^c K_2} \leq \gamma^{2d-c} <\gamma^d,\]
for the second claim.

\item
Continuing, since  $\alpha^c \leq \alpha^c \frac{1}{1-\beta^{d-c}}\leq \frac{1}{K_2}\leq \gamma^{2d}$, then  $1\leq \gamma^{-(2d-c)}\leq (\frac{\gamma}{\alpha})^c$, so $\gamma / \alpha \geq 1$, and thus
\begin{equation}\label{eq2estrellas} N \geq \frac{1}{2}\gamma^d \alpha^{-d}.\end{equation}
On the other hand, using the hypotheses, $c \in (0,d)$ and $\alpha \in (0,1)$,
\begin{equation}\label{eqappendix}
\alpha^d \leq \alpha^c \leq \frac{1}{K_2} (1- \beta^{d-c})\leq \frac{1}{K_2}|\log(\beta^{d-c})|= \frac{1}{K_2}(d-c) |\log(\beta)|,
\end{equation}
where in the last inequality we have used that $d-c \in (0,1)$, $\beta \in (0,\frac{1}{4}]$, $z:=\beta^{d-c} \in (0,1)$, and $f(z):=\log(\frac{1}{z})+z+1$ is a positive function on $(0,1)$, so $1-z \leq \log(\frac{1}{z})$.
Then, \begin{equation}\label{eq1estrella}N\alpha^d K_2 \leq N (d-c) |\log(\beta)|.\end{equation}
By inequalities \eqref{eq2estrellas} and \eqref{eq1estrella} and the definition of $K_2$,
\[N(d-c)|\log(\beta)|\geq N \alpha^d K_2 \geq \frac{\gamma^d}{2} K_2 \geq |\log(\gamma^d)|\]
which is equivalent to claim (3).
\end{enumerate}
This concludes the proof of Proposition \ref{cota card 2}.
\end{proof}

\subsubsection*{Conclusion of the proof}

\medskip
For each $\gamma \in (0,1)$ we proceed as follows:

By definition, $B_0 \in \mathcal{D}_0$. Moreover, $\phi_{0} (B_0):=0<(\gamma \rho)^c$, so $B_{0} \in \mathcal{D}'_{0}$.
We will construct a Cantor set $F$ as the intersection of a sequence of unions of closed sets:
\begin{itemize}
\item $\mathcal{B}_{0}:=\{B_{0}\}\subset \mathcal{D}'_{0}$.
\item Given a collection $\mathcal{B}_{j} \subset \mathcal{D}'_{jN}$ we construct the next level of sets $\mathcal{B}_{j+1} \subset \mathcal{D}'_{(j+1)N}$ by replacing each element of $B \in \mathcal{B}_{j}$ by $M:=\Big\lceil \beta^{-Nd} \left(\frac{1}{2^d 4^d \sqrt{d}^d} -3^d\gamma^d (1+4^d 2) \right)\Big\rceil$ elements of $\mathcal{D}_{(j+1)N}(B) \cap \mathcal{D}'_{(j+1)N}$; this is possible by Proposition \ref{canthijos}.
\end{itemize}
We define \[F:=\bigcap_{j \in \N_0} \bigcup_{B \in \mathcal{B}_j}B.\]
By a standard argument (see e.g. \cite[Example 4.6]{FalconerBook}),
\begin{align*}
\dim_{\rm H} (F)&\ \geq\  \frac{\log(M)}{|\log(\beta^{N})|}\ \geq\ \frac{\log(\beta^{-Nd})+ \log \left(\frac{1}{2^d 4^d \sqrt{d}^d} -3^d\gamma^d (1+4^d 2) \right)}{N |\log(\beta)|}\\
&=\ d+\frac{\log \left(\frac{1}{2^d 4^d \sqrt{d}^d} -3^d\gamma^d (1+4^d 2) \right)}{N |\log(\beta)|}\\
&\geq\  d-\frac{2\alpha^d \log \left(\left(\frac{1}{2^d 4^d \sqrt{d}^d} -3^d\gamma^d (1+4^d 2) \right)^{-1}\right)}{\gamma^d |\log(\beta)|},
\end{align*}
where we have used  \eqref{eq2estrellas}.

This last inequality holds for every $\gamma \in (0,1)$.
We can take, for example, $\gamma\in (0,1)$ such that $3^d \gamma^d (1+4^d 2)=\left(1-\frac{1}{2^d} \right) \frac{1}{(8\sqrt{d})^d}$ (this is not sharp, but it is close enough). For this $\gamma$ we get
\[\dim_{\rm H} (F) \geq d - K_1 \frac{\alpha^d}{|\log(\beta)|},\]
where
$$K_1:=\frac{2d (24\sqrt{d})^d \left(\log (2)+\log (8\sqrt{d})\right)}{1-\frac{1}{2^d}}\  \text{ making }\ K_2:=\left( \frac{(24 \sqrt{d})^d (1+4^d 2)}{1-\frac{1}{2^d}} \right)^2.$$

It remains to prove that $F \subset S \cap B_0$, since then
\[\dim_{\rm H} (S) \geq \dim_{\rm H} (S\cap B_{0}) \geq \dim_{\rm H} (F) \geq  d - K_1 \frac{\alpha^d}{|\log(\beta)|}.\]

Clearly $F \subset B_0$, by definition of $F$. We need to show that $F \subset S$. Let $x \in F$. For every $j \in \N$ there exists a unique $B_{jN}\in \mathcal{B}_j$ containing $x$.
By definition of $\mathcal{B}_{j+1}$ we have that $B_{(j+1)}\subset \frac{1}{2}B_{jN}$.  By \eqref{eqpi}, $\pi_{jN}(B_{(j+1)N})=B_{jN}$.
The sequence $(B_{jN})_{j}$ can be extended in a unique way to a sequence $(B_n)_{n}$ satisfying $B_n \in \mathcal{E}_n$ and $B_n:=\pi_n(B_{n+1})$ for all $n$. We interpret this sequence as Bob's moves, to which Alice responds according to her winning strategy.

Thus, for each $x \in F$ we construct a sequence $(B_n)_n$ as before, where $x$ is the only element of $\bigcap_n B_n$ (so $x=x_{\infty}$ is the outcome of the game). We will show that $x \in S$ by contradiction. Otherwise, suppose that $x \notin S$ where $S$ is an $(\alpha, \beta, c, \rho)$-winning set.
Then, $x \in \bigcup_{m \in \N_0} \bigcup_i A_{i,m}$, where $\sum_i (\diam A_{i,m})^c \leq (\alpha \rho_m)^c=(\alpha \beta^m \rho)^c$ (since it is a legal move for Alice we know that $\bigcup_i A_{i,m} \in \mathcal{A}(B_m)$).
So $x \in A \in \mathcal{A}(B_m)$ for some $m$, and as $x \in B_m$ we have $x \in A \cap B_m$.
Since  $\mathcal{A}^*_m(B_n)=\mathcal{A}(B_m)$ for every $n>m$ (because $\pi_m(B_n)=B_m$), then $\phi_j (B_{jN})\geq (\diam A)^c$ for every $j$ such that $jN>m$ (because $(\diam A)^c$ is just one term in the sum of the definition of $\phi_j (B_{jN})$ when $A \in \mathcal{A}^*_m(B_n)$).

On the other hand, since $B_{jN} \in \mathcal{D}'_j$, then $\phi_j(B_{jN})\leq (\gamma \rho_{jN})^c$. Putting everything together,  $\diam A \leq \gamma \rho_{jN}$ for all $j$ such that $jN>m$.
Letting $j \to \infty$, we get $\diam A=0$, a contradiction. So $x \in S$, that is $F \subset S$.

Finally, using  \eqref{eqappendix}, and that $K_1/K_2<1$,
\[K_1 \frac{\alpha^d}{|\log(\beta)|} \leq \frac{(d-c)K_1}{K_2}< d,\]
so
\[d-K_1 \frac{\alpha^d}{|\log(\beta)|}>0 \text{ if } \alpha^c \leq \frac{1}{K_2} (1-\beta^{d-c}).\]

This concludes the proof.
\end{proof}

\section*{Acknowledgements}
Alexia Yavicoli was financially supported by the Swiss National Science Foundation, grant n$^{\circ}$ P2SKP2\_184047.


\end{document}